\documentclass[reqno,12pt]{amsart}

\usepackage{amssymb}

\newcommand{\setheader}{Antibiotics Time Machine is NP-hard}
\newcommand{\setdate}{\today}
\newcommand{\settitle}{\setheader}

\usepackage{comment, color}
\usepackage{hyperref}
\usepackage[margin=1in]{geometry}
\usepackage[normalem]{ulem}

\usepackage{fancyhdr}
\fancypagestyle{title}{
\fancyhf{}
\fancyhead[L]{\setheader}
\fancyhead[R]{\setdate}
}

\renewcommand{\a}{\alpha}

\renewcommand{\c}{\mathbf{c}}

\renewcommand{\v}{\widetilde v}
\newcommand{\e}{\varepsilon}

\renewcommand{\d}{\delta}

\renewcommand{\l}{{\ell}}

\newcommand{\s}{\sigma}

\newcommand{\x}{\mathbf{x}}

\newcommand{\z}{\mathbf{z}}

\newcommand{\T}{{}^{\mathsf{T}}}
\renewcommand{\T}{\mathcal{T}}

\newcommand{\F}{\mathbf{F}}
\newcommand{\R}{\mathbb{R}}
\newcommand{\C}{\mathcal{C}}

\newcommand{\N}{\mathbb{N}}

\newcommand{\bs}{\smallsetminus}

\theoremstyle{plain}
\newtheorem{thm}{Theorem}
\newtheorem{lem}[thm]{Lemma}

\theoremstyle{definition}

\theoremstyle{remark}

\newcommand{\latin}[1]{\textsl{#1}}

\newcommand{\defn}[1]{\textbf{#1}}

\newcommand{\sTTs}{\s T_1 T_2\dotsm T_N\s^\top}

\begin{document}
\title{\settitle}
\author{Ngoc M Tran}
\address{Department of Mathematics, The University of Texas at Austin, TX 78751, USA}
\email{ntran@math.utexas.edu}
\author{Jed Yang}
\address{School of Mathematics, The University of Minnesota, MN 55455, USA}
\email{jedyang@umn.edu}

\maketitle
\begin{abstract}
The antibiotics time machine is an optimization question posed by Mira \latin{et al.} on the design of antibiotic treatment plans to minimize antibiotic resistance.
The problem is a variation of the Markov decision process.
These authors asked if the problem can be solved efficiently.
In this paper, we show that this problem is NP-hard in general.
\end{abstract}
\thispagestyle{empty}

\newcommand{\state}[1]{\mathbf{#1}}
\renewcommand{\s}{\state{s}}
\renewcommand{\d}{\state{d}}
\newcommand{\f}{\state{f}}

\newcommand{\Mira}{Mira \latin{et al.}}

\section{Problem statement and results}
Resistance to antibiotics in bacteria is a central problem in medicine.
One method for mitigating its consequences is to rotate the use of antibiotics.
A natural question arises: what is the optimal sequence of antibiotics one should apply?
In this work, we follow the model proposed by \Mira~\cite{mira2014rational}.
Suppose we have a population of bacteria, all of the same unmutated genotype (called the wild type).
We are given a set of antibiotics.
For each bacteria of a given type, an antibiotic mutates it to another type with some known probability.
The problem is to find a sequence of antibiotics (called a treatment plan)
that maximizes the fraction of bacteria returning to the wild type.
This is equivalent to minimizing the fraction of antibiotics-induced mutations in the bacterial population.

Here is a concrete description of the model.
Let $S$ be a set of $d$ states, where each state is a bacterial genotype.
A $d\times d$ matrix $T=(t_{ij})$ is a \defn{transition matrix} if $t_{ij}\in[0,1]$ and each row sums to~$1$.
Let $\T$ be a finite set consisting of $K$ transition matrices,
each corresponding to the effect of a given antibiotic on the genotypes of the bacterial population.
That is, $t_{ij}$ is the probability that a bacteria of genotype $i$ will have genotype $j$ after being treated with the corresponding antibiotic.
The standard $(d-1)$-dimensional simplex $\Delta^{d-1}$ is given by $\Delta^{d-1}=\{(x_1,x_2,\dotsc,x_d)\in\R^d\mid\text{$\sum_{i=1}^dx_i=1$ and $x_i\geq0$ for all $i$}\}$.
Let $\s=(1,0,\ldots,0)\in \Delta^{d-1}$ be the starting state of the population,
where all bacteria are of the first state (the wild type).
Let $N\in\N$ be the length of the treatment plan.
The \defn{time machine $\mathcal{TM}(d, K, N)$} is the solution to the following optimization problem

\begin{align}
\mbox{ maximize} \hspace{1em}  &\sTTs \nonumber \\
\mbox{ subject to} \hspace{1em} & T_1, \dotsc, T_N \in \T. \label{eqn:tm}
\end{align}

The term \emph{time machine}, coined by \Mira,
alludes to the idea of reversal to the wild type of unnecessary mutations done by antibiotic treatments.
They proposed and solved a specific numerical instance of the time machine for $16$ genotypes of the TEM $\beta$-lactamase,
with a specific model of mutation, and a set of up to six antibiotics.
Their algorithm is to enumerate all possible treatment plans of a given length with the given set of antibiotics
and output the optimal one.

With $K$ antibiotics, there are $K^N$ many treatment plans of length~$N$,
and thus direct enumeration is not an efficient algorithm.
\Mira\ raised the question of whether an efficient algorithm exists.
In this paper, we show that the general problem is NP-hard.

\begin{thm}\label{thm:tm}
The time machine optimization problem in \textup{(\ref{eqn:tm})} is NP-hard.
\end{thm}

We first consider an easier approximation problem: for a given threshold $\a\in[0,1]$,
decide whether there exists a treatment plan of length $N$ such that 
\begin{gather}
\sTTs \geq \a \nonumber \\
\mbox{ subject to } T_1, \dotsc, T_N \in \T. \label{eqn:tm.decision}
\end{gather}

Clearly if one can solve (\ref{eqn:tm}) in polynomial time,
then one can also solve (\ref{eqn:tm.decision}) in polynomial time.
Our main result states that the latter problem is NP-hard, and thus Theorem~\ref{thm:tm} follows as a corollary.

\begin{thm}
The time machine decision problem in \textup{(\ref{eqn:tm.decision})} is NP-hard.
\end{thm}

Our proof relies on reducing $3$-SAT to a special instance of (\ref{eqn:tm.decision}) where $\a=1$.
Since $3$-SAT is NP-hard \cite{karp1972reducibility}, we have that (\ref{eqn:tm.decision}) is also.
The reduction construction is described in Section~\ref{sec:3sat.reduction}.
The proof that this is a correct reduction is in Section~\ref{sec:3sat.proof}.

\subsection{Related literature}
The time machine may also arise in large scale robotics control,
where one has a large number of robots with $d$ possible states,
and each $T_i$ is an instruction which changes each robot's state independently at random.
In general, it can be viewed as a Markov decision problem with infinitely many states.
Indeed, if there was only one bacteria instead of a population, one may apply an antibiotic $T_i$,
check the new state of the bacteria, and repeat.
In this case, the time machine is a Markov decision process on $d$ states,
which can be solved by dynamic programming. When there is a population of bacteria,
the population state is a point in the standard $(d-1)$-dimensional simplex $\Delta^{d-1}$, which is an infinite set. 

There is a large literature on Markov decision processes,
see \cite{feinberg2002handbook} and references therein.
However, we are unaware of any existing results on the time machine,
or the equivalent of Theorem~\ref{thm:tm}.
Closely related to the time machine is the partially observable Markov decision process \cite{dutech2013partially},
where the underlying system is assumed to be in one of the unobservable $d$ states,
and $s \in \Delta^d$ reflects the uncertainty over the true state.
It would be interesting to see if techniques from this setup can yield polynomial time approximations to the time machine.

\subsection{Open problems}
Many interesting questions remain on the antibiotics time machine.
We mention two examples.
The last question was raised by Joe Kileel.
\begin{enumerate}
  \item For what classes of antibiotics $\mathcal{T}$ is the problem solvable in polynomial time?
  Are there biologically relevant instances?
  \item When does the limit $\lim_{N \to \infty}\max_{T_i \in \mathcal T}\sTTs$ exist?
  If so, how fast is the convergence?
  In this case, for large $N$, can one produce an approximate time machine in polynomial time?
\end{enumerate}

\section{Reduction construction} \label{sec:3sat.reduction}
Let $X=\{x_1,\dotsc,x_n\}$ be a set of Boolean variables taking Boolean values $\{+,-\}$.
A triple $(\e_ix_i,\e_jx_j,\e_kx_k)$ with $\e_i,\e_j,\e_k\in\{+,-\}$ is called a \defn{clause}.
Let $\C=\{c_1,\dotsc,c_m\}$ be a set of clauses.
Write $(x_i = v_i)_i$ to mean $x_i = v_i$ for $i = 1,\dotsc, n$.
An \defn{assignment} $(x_i=v_i)_i$ \defn{satisfies} a clause $c=(\e_ix_i,\e_jx_j,\e_kx_k)$
if $(\e_i,\e_j,\e_k)\neq(v_i,v_j,v_k)$.
Without loss of generality, assume each $x\in X$ actually occurs (negated or not) somewhere.
The $3$-SAT problem is: given a set of variables and clauses,
decide whether there exists an assignment that satisfies all clauses.
In this section, we reduce a $3$-SAT instance to an instance of (\ref{eqn:tm.decision}).

Let $N=m+2$ and $\a=1$.
Create $d=3n+m+3$ states and $K=7m+2$ transition matrices.
Create the start state~$\s$, a ``death'' state~$\d$, and a ``tally'' state~$\f$.
For each variable $x\in X$, create states $\x,\x^-,\x^+$.
For each clause $c=(\e_ix_i,\e_jx_j,\e_kx_k)\in\C$,
create state~$\c$.
We identify each state with its characteristic (row) vector and define each transition matrix by its action on these basis vectors and extend linearly.

\renewcommand{\S}{S}
\renewcommand{\F}{F}
\newcommand{\Tc}{T_c^{v_i,v_j,v_k}}
For each choice of $(v_i,v_j,v_k)\in\{+,-\}^3\bs\{(\e_i,\e_j,\e_k)\}$,
create a transition matrix $T=\Tc$ as follows:
\[ \z T=\begin{cases}
   \d & \z=\s \\
   \f & \z=\c \\
   \x_\l^{v_\l} & \z=\x_\l,\ \l\in\{i,j,k\} \\
   \d & \z=\x_\l^{-v_\l},\ \l\in\{i,j,k\} \\
   \z & \text{otherwise.}
\end{cases} \]
Define a starting matrix $\S$ by
\[ \z \S=\begin{cases}
   p\left(\sum_{x\in X}\x + \sum_{c\in\C}\c\right) & \z=\s \\
   \d & \text{otherwise,}
\end{cases} \]
where $p=1/(n+m)$.
Lastly, define a final matrix $\F$ by
\[ \z \F=\begin{cases}
   \s & \z\in\left\{\x_i^{v_i}:v_i\in\{+,-\}\right\} \\
   \s & \f \\
   \d & \text{otherwise.}
\end{cases}\]
Let $\mathcal{T}$ consist of $\S$, $\F$, and the $7m$ transition matrices $\Tc$ defined above.
This concludes the reduction construction.

\section{Proof of Main Theorem} \label{sec:3sat.proof}

We shall prove that with the setup in Section~\ref{sec:3sat.reduction},
the time machine decision problem (\ref{eqn:tm.decision}) solves the $3$-SAT problem with the given clauses.
Since the latter is NP-hard, this implies that (\ref{eqn:tm.decision}) is also.

We refer to the fraction of the bacterial population in a given state as its \defn{weight}.
First, suppose $(x_i=v_i)_i$ is a satisfying assignment.
Apply $\S$ to~$\s$.
For each $c=(\e_ix_i,\e_jx_j,\e_kx_k)\in\C$,
apply~$\Tc$,
which exists as $c$ is satisfied.
Finally, apply~$\F$.
\begin{lem}
The sequence of matrices given above sends all weights back to~$\s$.
That is, it is a solution of \textup{(\ref{eqn:tm.decision})} for $\alpha = 1$.
\end{lem}
\begin{proof}
For each $x\in X$,
the matrix $\S$ sends some weight to~$\x$.
The weight is moved to $\x^{v_i}$ the first time $\pm x$ occurs in a clause~$c$,
and is moved back to $\s$ by $\F$ at the end.
For each $c\in\C$,
the matrix $\S$ sends some weight to~$\c$.
This weight is moved to $\f$ by $\Tc$,
and is moved back to $\s$ by $\F$ at the end.
Therefore all the weights returns to~$\s$, as desired.
\end{proof}


Conversely,
suppose there is a sequence $T_1,T_2,\dotsc,T_N$ of transition matrices so that $$\sTTs=1.$$
The aim is to extract a satisfying assignment.

Consider the process of applying the transition matrices $T_i$ to $\s$ sequentially in $N$ steps.
Note that any weight at $\d$ stays at $\d$ forever.
As such, to achieve full weight at~$\s$ after $N$ steps,
the state $\d$ cannot receive weight at any point in the process.%
\footnote{Intuitively, $\d$ is a ``death'' state, where bacteria goes to die, never to be recovered to the wild type.}
Consequently, it is clear that the first matrix to apply has to be~$\S$,
as we have $T(\s)=\d$ for $T\in\T\bs\{\S\}$.

The only way for $\s$ to gain weight is to apply~$\F$.
Prior to applying $\F$ for the first time,
all weights must be supported on the $\x_i^\pm$ and~$\f$,
as to avoid losing any weight to the death state~$\d$.
In particular, for each clause $c=(\e_ix_i,\e_jx_j,\e_kx_k)\in\C$,
state~$\c$ must no longer carry weight at this point.
That is, after $\c$ receives some weight by~$\S$ in the first step,
it must subsequently lose the weight,
which can only be achieved by applying an associated matrix $\Tc$ for some choice of $(v_i,v_j,v_k)$.
This takes (at least) one step for each clause.
%
Since the sequence of matrices is of length precisely $N=m+2$,
we conclude that the sequence starts with $\S$, finishes with $\F$,
and contains precisely one matrix corresponding to each clause.

When $\S$ is applied,
the full weight at $\s$ is split into $n+m$ \defn{packets}, each of weight $p=1/(n+m)$.
Since no other matrices split weights,
we may consider the remaining process as a discrete system moving each packet as a unit.
We already saw that the $m$ packets associated to the clauses are moved to~$\f$ during the middle $m$ steps.
It remains to analyze the remaining $n$ packets associated to the variables.
To avoid moving any weight to the death state~$\d$,
the packet at $\x_i$ can only be moved to $\x_i^+$ or $\x_i^-$.
Once this happens, the packet can only be moved again by $\F$ at the last step.
So at the penultimate step,
there is a packet on $\x_i^{\v_i}$ for exactly a choice of~$\v_i$ for each~$i$.
The following lemma finishes the proof.

\begin{lem}
For each $i$, let $\v_i$ be such that $\x_i^{\v_i}$ has nonzero weight after $m+1$ steps,
\latin{i.e.},
$$\s T_1 T_2\dotsm T_{N-1}(\x_i^{\v_i})^\top>0.$$
Then $(x_i=\v_i)_i$ is a satisfying assignment.
\end{lem}

Indeed, consider a clause $c=(\e_ix_i,\e_jx_j,\e_kx_k)\in\C$.
We know that (exactly) one associated transition matrix $\Tc$ is used.
Suppose, towards a contradiction, that $\v_\l\neq v_\l$ for some $\l\in\{i,j,k\}$.
After applying $\Tc$,
the packet corresponding to $x_\l$ is at $\x_\l^{v_\l}$ or~$\d$,
with no way of moving to $\x_\l^{\v_\l}$, a contradiction.
So $(\v_i,\v_j,\v_k)=(v_i,v_j,v_k)\neq(\e_i,\e_j,\e_k)$ by construction,
implying that clause~$c$ is satisfied.
\qed

This shows that a $3$-SAT instance has a solution if and only if the associated time machine can attain a threshold of~$1$.
We therefore conclude that the time machine decision problem is NP-hard, as desired.

\subsection*{Acknowledgements}
The authors wish to thank Bernd Sturmfels and Joel Kileel for stimulating discussions.
The first-named author is supported by an award from the Simons Foundation (\#197982 to The University of Texas at Austin).
The second-named author is supported by NSF RTG grant NSF/DMS-1148634.

\bibliographystyle{alpha}
\bibliography{bacteria}

\end{document}